\documentclass[a4paper,10pt]{amsart}
\usepackage[T1]{fontenc}
\usepackage[utf8]{inputenc}
\usepackage{amsmath,amsthm,amsfonts,amssymb,latexsym}
\usepackage{amsmath}

\def\la#1{\hbox to #1pc{\leftarrowfill}}
\def\ra#1{\hbox to #1pc{\rightarrowfill}}
\def\fract#1#2{\raise4pt\hbox{$ #1 \atop #2 $}}
\def\bbS{{S}_{n}}
\def\bbs{{ S_{n}^{\ast}}}
\def\H2{\hbox{{\rm H}\kern-0.9em\hbox{{\rm I}\ \ }}^2}
\def\h{\hbox{{\rm H}\kern-0.9em\hbox{{\rm I}\ \ }}}
\def\T{\hbox{{\rm T}\kern-0.66em\hbox{{\rm I}\ \ }}}
\def\R{\hbox { {\rm R}\kern-0.9em\hbox{{\rm I}\ }}}
\def\C{\hbox { {\rm C}\kern-0.56em\hbox{{\rm I}\ }}}

\theoremstyle{definition}
\theoremstyle{remark}
\numberwithin{equation}{section}

\newtheorem{theorem}{\textbf{Theorem}}[section]
\newtheorem{corollary}[theorem]{\textbf{Corollary}}
\newtheorem{lemma}[theorem]{\textbf{Lemma}}

\begin{document}
\title{On the higher rank numerical range of the shift operator}
\author{Haykel GAAYA}
\address{\ddag.Institute Camille Jordan, Office 107 University of Lyon1,
43 Bd November 11, 1918, 69622-Villeurbanne, France.}
\email{\ddag  gaaya@math.univ-lyon1.fr }
\maketitle
\begin{abstract}
For any n-by-n complex matrix T and any $1\leqslant k\leqslant n$, let $\Lambda_{k}(T)$ the set of all $\lambda\in \C$ such that $PTP=\lambda P$ for some rank-k orthogonal projection $P$ be its higher rank-k numerical range. It is shown that if $\bbS$ is the n-dimensional shift on ${\C}^{n}$ then its rank-k numerical range is the circular disc centred in zero and with radius $\cos\dfrac{k\pi}{n+1}$ if $1<k\leqslant\left[\frac{n+1}{2} \right]$ and the empty set if $\left[\frac{n+1}{2} \right]<k\leqslant n$, where $\left[x \right] $ denote the integer part of $x$. This extends and rafines previous results of U. Haagerup, P. de la Harpe \cite{Haagerup} on the classical numerical range of the n-dimensional shift on${\C}^{n}$. An interesting result for higher rank-$k$ numerical range of nilpotent operator is also established.
\end{abstract}

\section{Introduction}
\hspace{0.5cm}Let $\mathcal{H}$ be a complex separable Hilbert space and $\mathcal{B(H)}$ the collection of all bounded linear operator on $\mathcal{H}$. The numerical range of an operators $T$ in $\mathcal{B(H)}$ is the subset $$W(T)=\left\lbrace <Tx,x>\in\C;x\in \mathcal{H} ,\lVert x \lVert\leqslant 1\right\rbrace $$  of the plane, where $<.,.>$ denotes the inner product in $\mathcal{H}$ and the numerical range of $T$ is defined by $$ \omega_{2}(T)=\sup \left\lbrace \lvert z\lvert; z\in W(T) \right\rbrace .$$ 
\hspace{0.5cm}We denote by $S$ the unilateral shift acting on the Hardy space $\H2$ of the square summable analytic functions. 
$$\begin{array}{ccccc}
S & : & \H2 & \to & \H2 \\
& & f & \mapsto & zf(z) \\
\end{array}$$
\hspace{0.5cm}Beurling's theorem implies that the non zero invariant subspaces of $S$ are of the forme $\phi~\H2$, where $\phi$ is some inner function . Let $S(\phi)$ denote the compression of $S$ to the space $ H(\phi)=\H2\ominus\phi~\H2 $ :
$$S(\phi)f(z)=P(zf(z)),$$ where $P$ denotes the ortogonal projection from $\H2$ onto $ H(\phi)$. The space $H(\phi)$ is a finite-dimensional exactly when $\phi$ is a finite Blaschke product. The numerical radius and numerical range of the model operator $S(\phi)$ seems to be important and have many applications. In \cite{Cassier}, Badea and Cassier showed that there is relationship between numerical radius of $S(\phi)$ and Taylor coefficients of positive rational functions on the torus and more recently in \cite{Gaaya}, the author gave an extension of this result. However the evaluation of the numerical radius of $S(\phi)$ under an explicit form is always an open problem. The reader may consult \cite{Gaaya} for an estimate of $S(\phi)$ where $\phi$ is a finite Blashke product with unique zero.
In the particular case where $\phi(z)=z^{n}$, $S(\phi)$ is unitarily equivalent to $\bbS$ where 

$$
\bbS=\left(
\begin{array}{cccc}
0 &       &       &  \\
1 & \ddots &      &   \\
  & \ddots &\ddots    &  \\
  &        & 1    & 0 
  
\end{array}
\right)
.$$
\hspace{0.5cm}In \cite{Haagerup}; it is proved that $W(\bbS)$ is the closed disc $D_{n }=\left\lbrace z\in \C ; \lvert z\lvert \leqslant \cos\frac{\pi}{n+1}\right\rbrace$ and $\omega_{2}(\bbS)=\cos\frac{\pi}{n+1}$
and more general
\begin{theorem}[\cite{Haagerup}]
\textit{ Let $T$ be an operator on $\mathcal{H}$ such that $T^{n}=0$ for some $n \geq 2$. One has:$$\omega_{2}(T)\leqslant \lVert T\lVert \cos\frac{\pi}{n+1}$$
and $\omega_{2}(T)=\lVert T\lVert \cos\frac{\pi}{n+1}$ when $T$ is unitarily equivalent to $\lVert T\lVert S_{n}$}.
\end{theorem}
\hspace{0.5cm}In this mathematical note, we extend this result to the higher rank-$k$ numerical range of the shift. The notion of the the higher rank-k numerical range of $T \in\mathcal{B(H)}$ is introduced in \cite{Choi3} and it's denoted by:
\begin{eqnarray*}
 \Lambda_{k}(T)=\left\lbrace \lambda\in \C: PTP=\lambda P ~ \mbox{for some rank-$k$ orthogonal projection} ~P\right\rbrace, 
\end{eqnarray*}
The introduction of this notion was motivated by a problem in quantum error correction; see \cite{Choi4}. If $P$ is a rank-1 orthogonal projection then $P=x\otimes x$ for some $x\in\C^{n}$ and $PTP=<Tx,x>P$. Then when $k=1$, this concept is reduces to the classical numerical range $W(T)$, which is well known to be convex by the Toeplitz-Hausdorff theorem; for exemple see \cite{Li1} for a simple proof. In \cite{Choi1}, it's conjectured that $\Lambda_{k}(T)$ is convex, and reduced the convexity problem to the problem of showing that $0\in\Lambda_{k}(T') $ where $$T'=\left(
\begin{array}{cc}
I_{k} &   X     \\
 Y&   -I_{k} \end{array}
\right) $$
for arbitrary $X,Y\in \mathcal{M}_{k}$ (the algebra of $k\times k$ complex matrix). They further reduced this problem to the existance of a Hermitian matrix $H$ satisfying the matrix equation 
\begin{equation}
I_{k}+MH+H{M}^{\ast}-HRH=H
\end{equation}
for arbitrary $M\in \mathcal{M}_{k}$ and a positive definite $R\in \mathcal{M}_{k}$. In \cite{Woerdeman}, H. Woerdeman proved that equation (1.1) is equivalent to Ricatti equation:
\begin{equation}
HRH-H({M}^{\ast}-I_{k}/2)-(M-I_{k}/2)H-I_{k}=0_{k},
\end{equation}
and using the theory of Ricatti equations (see \cite{Lancaster}, Theorem 4), the equation (1.2) is solvable which prove the convexity of $\Lambda_{k}(T)$. In \cite{Choi3}, the authors showed that if dim$\mathcal{H}<\infty$ and $T \in \mathcal{B(H)}$ is a Hermitian matrix with eigenvalues $\lambda_{1}\leqslant\lambda_{2}\dots\leqslant\lambda_{n}$ then the rank-k nuemrical range $\Lambda_{k}(T) $ coincides with $\left[\lambda_{k},\lambda_{n+1-k} \right] $ which is a non-degenerate closed interval if $\lambda_{k}<\lambda_{n+1-k}$, a singleton set if $\lambda_{k}=\lambda_{n+1-k}$ and an empty set if $\lambda_{k}>\lambda_{n+1-k}$. In \cite{Lie4}, the authors proved that if dim$\mathcal{H}=n$
$$\Lambda_{k}(T)=\bigcap_{\theta\in[0,2\pi[ }\left\lbrace \mu\in\C:e^{i\theta}\mu+e^{-i\theta}\overline{\mu}\leqslant \lambda_{k}\left( e^{i\theta}T+e^{-i\theta}{T}^{\ast}\right) \right\rbrace,$$ for $1\leqslant k\leqslant n$, where $\lambda_{k}(H)$ denote the $k$th largest eigenvalue of the hermitian matrix $H\in \mathcal{M}_{n}$. This result establishes that if dim$\mathcal{H}=n$ and $T \in \mathcal{B(H)}$ is a normal matrix with eigenvalues $\lambda_{1},\dots,\lambda_{n}$ then $$\Lambda_{k}(T)=\bigcap_{1\leqslant j_{1}<\dots<j_{n-k+1}\leqslant n }\mbox{conv}\left\lbrace\lambda_{j_{1}},\dots,\lambda_{j_{n+1-k}} \right\rbrace .$$
\hspace{0.5cm}We close this section by the following properties wich are easly checked. The reader may consult \cite{Choi1},\cite{Choi2},\cite{Choi3},\cite{Choi4},\cite{Gau} and \cite{Li2}.
\begin{itemize}
\item[P1.]For any $a$ and $b \in\C, \Lambda_{k}(aT+bI)=a\Lambda_{k}(T)+b.$
\item[P2.]$\Lambda_{k}(T^{\ast})=\overline{\Lambda_{k}(T)}.$
\item[P3.]$\Lambda_{k}(T\oplus S)\supseteq\Lambda_{k}(T)\cup\Lambda_{k}(S).$
\item[P4.]For any unitary $U\in\mathcal{B(H)}, \Lambda_{k}(U^{\ast}TU)=\Lambda_{k}(T).$
\item[P5.]If $T_{0}$ is a compression of $T$ on a subspace $\mathcal{H}_{0}$ of $\mathcal{H}$ such that dim$\mathcal{H}_{0}\geq k$, then $\Lambda_{k}(T_{0})\subseteq\Lambda_{k}(T).$
\item[P6.]$W(T)\supseteq\Lambda_{2}(T)\supseteq\Lambda_{3}(T)\supseteq\dots~.$
\end{itemize}
Some results from \cite{Cassier} will be also developed in this context in a forthcoming paper.
\section{main theorem}
\hspace{0.5cm}In the following theorem we give the higher rank-$k$ numerical range of the n-dimensional shift on $\C^{n}$.
\begin{theorem}
\textit{ For any $n\geq2$ and $1\leqslant k\leqslant n,~ \Lambda_{k}(\bbS)$ coincides with the circular disc $\{z\in\C:\lvert z\lvert\leqslant\cos\dfrac{k\pi}{n+1}\}$ if $1\leqslant k\leqslant\left[\frac{n+1}{2} \right]$ and the empty set if $\left[\frac{n+1}{2} \right]<k\leqslant n$.}
\end{theorem}
\begin{proof}
First observe that
\begin{eqnarray}
 \Lambda_{k}(\bbS)&=&\bigcap_{\theta\in[0,2\pi[ }\left\lbrace \mu\in\C:e^{i\theta}\mu+e^{-i\theta}\overline{\mu}\leqslant \lambda_{k}\left( e^{i\theta}\bbS+e^{-i\theta}\bbs\right) \right\rbrace\nonumber\\
&=&\bigcap_{\theta\in[0,2\pi[ }\left\lbrace \mu\in\C:Re(e^{i\theta}\mu)\leqslant \dfrac{1}{2}\lambda_{k}\left( e^{i\theta}\bbS+e^{-i\theta}\bbs\right) \right\rbrace\nonumber\\
&=&\bigcap_{\theta\in[0,2\pi[ }e^{i\theta}\left\lbrace z\in\C:Re(z)\leqslant \dfrac{1}{2}\lambda_{k}\left( e^{i\theta}\bbS+e^{-i\theta}\bbs\right) \right\rbrace
\end{eqnarray}
On the other hand, we have
$$e^{i\theta}\bbS+e^{-i\theta}\bbs=\left(
\begin{array}{cccccc}
0                & e^{-i\theta}          &  0          &\dots    &     0          &  0   \dots\\
e^{i\theta}      & 0                     &e^{-i\theta} & \dots    &     0           & 0    \dots               \\
0                & e^{i\theta}           &  0          & \dots   &0           &    0     \\
 \vdots               &  \vdots                    &   \vdots        & \ddots   &     \vdots    &   \vdots    \\
0           &0                  &0        & \dots &0      &   e^{-i\theta}         \\
  0               &     0                  &      0       &   \dots    &  e^{i\theta} &0
\end{array}
\right).
$$
Note that $e^{i\theta}\bbS+e^{-i\theta}\bbs$ is a Toeplitz matrix associated to the Toeplitz form $$f_{\theta}(t)=2\cos(\theta+t).$$
The eigenvalues satisfy the caracteristic equation
\begin{eqnarray}
 \Delta_{n}(\lambda)&=& Det \left( e^{i\theta}\bbS+e^{-i\theta}\bbs\right) \nonumber\\
&=&\left\lvert
\begin{array}{cccccc}
-\lambda                & e^{-i\theta}          &  0          &\dots    &     0          &  0   \dots\\
e^{i\theta}      & -\lambda                     &e^{-i\theta} & \dots    &     0           & 0    \dots               \\
0                & e^{i\theta}           &  -\lambda           & \dots   &0           &    0     \\
 \vdots               &  \vdots                    &   \vdots        & \ddots   &     \vdots    &   \vdots    \\
0           &0                  &0        & \dots &-\lambda       &   e^{-i\theta}         \\
  0               &     0                 &      0       &   \dots    &  e^{i\theta} &-\lambda 
\end{array}
\right\lvert\nonumber
\end{eqnarray}

Expanding this determinant, we obtain the recurrence relation
$$\Delta_{n}(\lambda)=-\lambda\Delta_{n-1}-\Delta_{n-2},~~~n=2,3,4,\dots,$$
This recurrence relation holds also for $n=1$ provided we put $\Delta_{0}=1$ and $\Delta_{-1}=0$. In order to find an explicit representation of $\Delta_{n}(\lambda)$, we write convenently $$ \lambda=2\cos(\theta+t)=f_{\theta}(t)$$
and form the caracteristic equation $$\rho^{2}=-\lambda\rho-1=-2\rho\cos(\theta+t)-1 $$ 
with the roots $-e^{i(\theta+t)}$ and $-e^{-i(\theta+t)}$ so that $$\Delta_{n}(2\cos(\theta+t))=(-1)^{n}(Ae^{in(\theta+t)}+Be^{-in(\theta+t)})$$ where the constants $A$ and $B$ can be determined from the cases $n=-1$ and $n=0$. Thus $$\Delta_{n}(2\cos(\theta+t))=(-1)^{n}\frac{\sin((n+1)(\theta+t))}{\sin(\theta+t).}$$
This yields the eigenvalues $$\lambda_{\nu}=2\cos(\dfrac{\nu\pi}{n+1}),~~~ \nu=1,2,\dots n.$$
This implies of course that 
$$\Lambda_{k}(\bbS)=\bigcap_{\theta\in[0,2\pi[ }e^{i\theta}\left\lbrace z\in\C:Re(z)\leqslant \cos(\dfrac{k\pi}{n+1}) \right\rbrace$$
Thus $\Lambda_{k}(\bbS)$ is the intersection of closed half planes. We note that $\cos(\dfrac{k\pi}{n+1})$ is positive if and only if $k\leqslant\left[\frac{n+1}{2} \right]$. 

\textbf{Case 1.}If $k\leqslant\left[\frac{n+1}{2} \right]$ In this case $\Lambda_{k}(\bbS)$ is circular disc $\{z\in\C:\lvert z\lvert\leqslant\cos\dfrac{k\pi}{n+1}\}$. 

\textbf{Case 2.} If  $k>\left[\frac{n+1}{2} \right]$, then 
\begin{eqnarray*}
\Lambda_{k}(\bbS)&\subseteq& \left\lbrace z\in\C:Re(z)\leqslant \cos(\dfrac{k\pi}{n+1}) \right\rbrace \bigcap e^{i\pi}\left\lbrace z\in\C:Re(z)\leqslant \cos(\dfrac{k\pi}{n+1}) \right\rbrace\\
&=&\emptyset.
\end{eqnarray*}

This completes the proof.
\end{proof}
\hspace{0.5cm}On the sequel of this paper, let denote by
$$\rho(k,r) = \left\{
    \begin{array}{ll}
        k/r & \mbox{if} ~k/r~ \mbox{is is integer} \\
        \left[ k/r\right] +1 & \mbox{unless}
    \end{array}
\right.
$$where $k$ and $r$ are arbitrary numbers.
\begin{lemma}
 \textit{For a fixed $n\geq1$ and $r\geq1$, let denote by $\lambda_{1}>\dots>\lambda_{n}$; $n$ real numbers and ${({\lambda'}_{p})}_{1\leqslant p\leqslant nr}$ a finite sequence defined by: $${\lambda'}_{1}=\dots={\lambda'}_{r}=\lambda_{1},\dots,{\lambda'}_{(n-1)r+1}=\dots={\lambda'}_{nr}=\lambda_{n}.$$
Then for each $1\leqslant k\leqslant nr$, the $k$th largest term of ${({\lambda'}_{t})}_{1\leqslant t\leqslant nr}$ is $\lambda_{\rho(k,r)}$.}
\end{lemma}
\begin{proof}
The claim is obvious in the case where $r=1$. We may assume $r\geq 2$. We prove the result by induction on $k$.
If $k=1$, then the largest term is $\lambda_{1}=\lambda_{\rho(1,r)}$. So the result hold for $k=1$. Assume that $k>1$, and the reslut is valid for the $m$th largest term of ${({\lambda'}_{t})}_{1\leqslant t\leqslant nr}$ whenever $m<k$.

\textbf{Case 1.} Suppose that $\rho(k-1,r)=\frac{k-1}{r}$, then there exists $1\leqslant p\leqslant n-1$ such that $k-1=pr$. By induction assumption, we have $\lambda_{\rho(k-1,r)}={\lambda'}_{pr}=\lambda_{p}$, which implies that the $k$th largest term of ${({\lambda'}_{t})}_{1\leqslant t\leqslant nr}$ is $${\lambda'}_{pr+1}=\lambda_{p+1}=\lambda_{\frac{k-1}{r}+1}=\lambda_{[\frac{k}{r}]+1}=\lambda_{\rho(k,r)}.$$

\textbf{Case 2.} Suppose that $\rho(k-1,r)=[\frac{k-1}{r}]+1$, then there exist $1\leqslant q\leqslant n-1$ and $1\leqslant s\leqslant r-1$ such that $k-1=qr+s$. First, note that $\rho(k-1,r)=\rho(k,r)$. On the other hand, by induction assumption, we have $\lambda_{\rho(k-1,r)}={\lambda'}_{qr+s}=\lambda_{q+1}$. Consequently the $k$th largest term of ${({\lambda'}_{t})}_{1\leqslant t\leqslant nr}$ is $${\lambda'}_{qr+s+1}={\lambda}_{q+1}=\lambda_{\rho(k-1,r)}=\lambda_{\rho(k,r)}.$$ The proof is now complete.
\end{proof}

\hspace{0.5cm}Let $D_{T}=(I_{N}-T^{\ast}T)^{1/2}$ be the defect operator of $T$ and $\mathcal{D}_{T}$ the closed range of $D_{T}$. Let denote by $r=\mbox{dim} \mathcal{D}_{T}$.
\begin{theorem}
\textit{Consider $T\in \mathcal{B(H)}$ such that $\lVert T\lVert\leqslant 1$ and $T^{n}=0$. Then $\Lambda_{k}(T)$ is contained in the circular disc $\{z\in \C:\lvert z\lvert\leqslant\cos(\frac{\rho(k,r)\pi}{n+1})\}$ if $1\leqslant\rho(k,r)\leqslant[\frac{n+1}{2}]$ and empty if $\rho(k,r)>[\frac{n+1}{2}]$.}
\end{theorem}
\begin{proof} 
If $T$ is a contaction with $T^{n}=0$, then $T$ can be viewed as a compression of $I_{r}\otimes \bbs$ acting on the Hilbert space $\mathcal{D}_{T}\otimes{\C}^{n}$. Consider the isometry $\mathcal{H}\rightarrow \mathcal{D}_{T}\otimes{\C}^{n}$, $$V(x)={\sum}_{t=1}^{n}D_{T}T^{t-1}x\otimes e_{t}$$
where ${\left\lbrace e_{l}\right\rbrace}_{l=1}^{n}$ is the canonical basis of ${\C}^{n}$. Note that $$VTx={\sum}_{t=1}^{n}D_{T}T^{t}x\otimes e_{t}={\sum}_{t=1}^{n-1}D_{T}T^{t}x\otimes e_{t}=(I_{r}\otimes \bbs)Vx.$$ It follows that $$T=V^{\ast}(I_{r}\otimes \bbs)V$$
and from (P.5) 
\begin{equation}
 \Lambda_{k}(T)=\Lambda_{k}(V^{\ast}(I_{r}\otimes \bbs)V)\subseteq\Lambda_{k}(I_{r}\otimes \bbs),~~\mbox{for any}   ~1\leqslant k\leqslant nr.
\end{equation}
Now,\\ 
$ \Lambda_{k}(I_{r}\otimes \bbs)$
\begin{eqnarray*}
 &=&\bigcap_{\theta\in[0,2\pi[ }\left\lbrace \mu\in\C:e^{i\theta}\mu+e^{-i\theta}\overline{\mu}\leqslant \lambda_{k}\left( e^{i\theta}(I_{r}\otimes \bbs)+e^{-i\theta}(I_{r}\otimes \bbs)^{\ast}\right) \right\rbrace\\
 &=&\bigcap_{\theta\in[0,2\pi[ }\left\lbrace \mu\in\C:e^{i\theta}\mu+e^{-i\theta}\overline{\mu}\leqslant \lambda_{k}\left( e^{i\theta}(I_{r}\otimes \bbs)+e^{-i\theta}(I_{r}\otimes {\bbS})\right) \right\rbrace\\
 &=&\bigcap_{\theta\in[0,2\pi[ }\left\lbrace \mu\in\C:e^{i\theta}\mu+e^{-i\theta}\overline{\mu}\leqslant \lambda_{k}\left(I_{r}\otimes( e^{i\theta}\bbS+e^{-i\theta}\bbs)\right) \right\rbrace\\
 &=&\bigcap_{\theta\in[0,2\pi[ }\left\lbrace \mu\in\C:e^{i\theta}\mu+e^{-i\theta}\overline{\mu}\leqslant \lambda_{k}\left(\oplus_{i}^{r}( e^{i\theta}\bbS+e^{-i\theta}\bbs)\right) \right\rbrace\\
 &=&\bigcap_{\theta\in[0,2\pi[ }e^{i\theta}\left\lbrace z\in\C:Re(z)\leqslant \cos(\frac{\rho(k,r)\pi}{n+1})\ \right\rbrace\\
\end{eqnarray*}

where the last equality is due to the lemma (2.2) and theorem (2.1). 
Thus $$\Lambda_{k}(I_{r}\otimes \bbs)= \left\{
    \begin{array}{ll}
        \overline{\textit{D}(0,\cos(\frac{\rho(k,r)\pi}{n+1}))} & \mbox{if }~ 1\leqslant\rho(k,r)\leqslant[\frac{n+1}{2}] \\
        \emptyset & \mbox{if} ~~[\frac{n+1}{2}]<\rho(k,r)\leqslant n
    \end{array}
\right.$$
Therefore,\\ if $1\leqslant k\leqslant nr$, (2.2) implies that $\Lambda_{k}(T)\subseteq \overline{\textit{D}(0,\cos(\frac{\rho(k,r)\pi}{n+1}))}$ if $1\leqslant\rho(k,r)\leqslant[\frac{n+1}{2}]$ and empty if $[\frac{n+1}{2}]<\rho(k,r)\leqslant n$. Finally, if $k>nr$, $\Lambda_{k}(T)=\emptyset$ from (P6).
\end{proof}
\begin{corollary}[U. Haagerup, P. de la Harpe,\cite{Haagerup}]
 \textit{Consider $T \in \mathcal{B(H)}$ such that $\lVert T\lVert\leqslant 1$ and $T^{n}=0$. Then we have $\omega_{2}(T)\leqslant\cos(\frac{\pi}{n+1}).$}
\end{corollary}
\begin{proof}
  $T=V^{\ast}(I_{r}\otimes \bbs)V$ where $V:H\rightarrow \mathcal{D}_{T}\otimes{\C}^{n}$,$$V(x)={\sum}_{t=1}^{n}D_{T}T^{t-1}x\otimes e_{t}.$$ Now $$W(T)=\Lambda_{1}(T)=\Lambda_{1}(V^{\ast}(I_{r}\otimes \bbS)V)\subseteq\Lambda_{1}(I_{r}\otimes \bbS)=\overline{\textit{D}(0,\cos\frac{\pi}{n+1})}.$$
\end{proof}

\underline{Acknowledgements:} The author would like to express his gratitude to Gilles Cassier for his help and his good advices.


\begin{thebibliography}{99}
\bibitem{Cassier}  C. Badea and G. Cassier, Constrained von neumann inequalities, Adv. Math. 166 (2002), no. 2, 260--297.
\bibitem{Choi1}  M.-D. Choi, M. Giesinger, J. A. Holbrook, and D. W. Kribs, Geometry of higher-rank numerical ranges, Linear and Multilinear Algebra 56 (2008), 53-64.
\bibitem{Choi2}  M.-D. Choi, J. A. Holbrook, D.W. Kribs, and K. Zyczkowski, Higher-rank numerical ranges of unitary and normal matrices, Operators and Matrices 1 (2007), 409-426. 
\bibitem{Choi3}  M.-D. Choi, D. W. Kribs, and K. Zyczkowski, Higher-rank numerical ranges and compression
problems, Linear Algebra Appl. 418 (2006), 828-839. 
\bibitem{Choi4}  M.-D. Choi, D. W. Kribs, and K. Zyczkowski, Quantum error correcting codes from the
compression formalism, Rep. Math. Phys. 58 (2006), 77-91. 
\bibitem{Gaaya} H. Gaaya, On the numerical radius of the truncated adjoint Shift. to appear.
\bibitem{Gau} H.L. Gau, C.K. Li, P.Y. Wu, Higher-rank numerical ranges and dilations, J. Operator Theory, in press.
\bibitem{Haagerup}  U.Haagerup and P. dela Harpe, The numerical radius of a nilpotent operator on a Hilbert space, Proc. Amer. Math. Soc. 115(1992), 371--379.
\bibitem{Lancaster}  P. Lancaster and L. Rodman, Algebraic Riccati equations, Oxford Science Publications,
The Clarendon Press, Oxford University Press, New York, 1995. 
\bibitem{Li1}  C.-K. Li, A simple proof of the elliptical range theorem, Proc. Amer. Math. Soc. 124 (1996),
1985-1986. 
\bibitem{Li2}  C.-K. Li, Y. T. Poon and N.-S. Sze, Condition for the higher rank numerical range to be
non-empty, Linear and Multilinear Algebra, to appear.
\bibitem{Li3}  C.-K. Li, Y. T. Poon and N.-S. Sze, Higher rank numerical ranges and low rank perturbations
of quantum channels, preprint. http://arxiv.org/abs/0710.2898.
\bibitem{Lie4}  C.-K. Li, N.-S. Sze, Canonical forms, higher rank numerical ranges, totally isotropic subspaces, and matrix equations, Proc. Amer. Math. Soc. Volume 136, Number 9, September 2008, Pages 3013–3023 .
\bibitem{Pop}  C. Pop, On a result of Haagerup and de la Harpe. Rev. Roumaine Math. Pures Appl. 43 (1998), no. 9-10, 869--871.
\bibitem{Stewart}  G. W. Stewart and J.-G. Sun, Matrix Perturbation Theory, Academic Press, New York,
1990. 
\bibitem{Woerdeman}  H. Woerdeman, The higher rank numerical range is convex, Linear and Multilinear Algebra
56 (2008), 65-67.

\end{thebibliography}
\end{document}